\documentclass[fleqn, 11pt]{frank}

\usepackage{amsmath, amscd, amsthm, amsfonts, amssymb}
\usepackage{mathrsfs, bbm, cancel, stmaryrd, array, comment}
\usepackage[FIGTOPCAP]{subfigure}
\usepackage[all]{xy}
\usepackage{graphicx,enumitem}
\usepackage{scrtime} 
\usepackage{rotating, lscape, wrapfig, tabularx, color, microtype, xfrac, textcomp}
\usepackage[hyperfootnotes=false]{hyperref}
\usepackage{color}
\definecolor{darkblue}{rgb}{0,0,1}
\hypersetup{pdfpagemode=UseNone, colorlinks=true, breaklinks=true, linkcolor=blue, menucolor=black, urlcolor=blue, citecolor=red}
\usepackage[bottom]{footmisc}
\usepackage{datetime}

\newcommand{\RR}{\mathbbm{R}}
\newcommand{\CC}{\mathbbm{C}}
\newcommand{\NN}{\mathbbm{N}}

\newtheorem{proposition}{Proposition}%[section] 
\newtheorem{theorem}[proposition]{Theorem}

\theoremstyle{remark}

\theoremstyle{definition}

\newtheorem{remark}[proposition]{Remark}

\newcommand{\bfx}{{\boldsymbol\alpha}}
\newcommand{\bfr}{{\mathbf{r}}}
\newcommand{\bfy}{{\boldsymbol\beta}}
\newcommand{\va}{{\mathbf{a}}}

\allowdisplaybreaks
\begin{document}

\sloppy \raggedbottom

\title{On the characterization of the degree of interpolation polynomials in terms of certain combinatorical matrices}

\begin{start}

	\author{Frank Klinker}{1},\ 
	\author{Christoph Reineke}{2}\\

	\address{ %Fakult\"at f\"ur Mathematik, TU Dortmund, 44221 Dortmund\\[0.5ex]
	\href{mailto:frank.klinker@math.tu-dortmund.de}{frank.klinker@math.tu-dortmund.de} (Corresponding author)\\}{1}

	\address{\href{mailto:christoph_reineke@gmx.de}{christoph\_reineke@gmx.de}\\}{2}

\noindent
\makebox[0.8\textwidth]{%	
\begin{minipage}{0.85\textwidth}
\begin{Abstract}
In this note we show that the degree of the interpolation polynomial for equidistant base points is characterized by the regularity of matrices of combinatorical type. 
	\end{Abstract}
\end{minipage}}

\renewcommand{\dateseparator}{-}
\let\thefootnote\relax\footnote{Faculty of Mathematics, TU Dortmund, 44221 Dortmund}
\let\thefootnote\relax\footnote{Published in {\em Arab.~J.~Math.}  \href{https://doi.org/10.1007/s40065-018-0212-x}{https://doi.org/10.1007/s40065-018-0212-x} (online 2018)} %\ -- \currenttime}

\end{start}	

%\tableofcontents

\runningheads{%
{\small\sf Preprint}\hspace{19.1em}
Determinants and Interpolation
}{% 
{\small\sf Preprint}\hspace*{23.2em}   
F.~Klinker, C.~Reineke
}

\section{Statement of the main result}

Let $A_{s,\va}\in M_{\ell+1}\RR$ be the matrix depending on a vector  $\va=(a_0,\ldots,a_\ell)\in\RR^{\ell+1}$ and an integer $s\geq 0$ and defined by 
\begin{equation}\label{matrixA}
A_{s,\va}:=\begin{pmatrix}
1^{\ell-1} & 2^{\ell-1}&\cdots& (\ell+1)^{\ell-1}\\
(\ell+2)^{\ell-1} & (\ell+3)^{\ell-1}&\cdots& (2\ell+2)^{\ell-1}\\
\vdots &\vdots &&\vdots \\
(\ell^{2})^{\ell-1} & (\ell^{2}+1)^{\ell-1} &\cdots& (\ell^2+\ell)^{\ell-1}\\
0^sa_0 &1^sa_1&\cdots &\ell^sa_\ell
\end{pmatrix}\,.
\end{equation}
Furthermore, denote by $q_\va$ the interpolation polynomial of degree at most $\ell$ defined by the base points $x_0,\ldots,x_\ell$ and the value vector $\va$:
\begin{equation}\label{2}
q_\va(x_i)=a_i \,,\quad i=0,\ldots,\ell\,.
\end{equation}
In this text we show the following remarkable characterization of the degree of the interpolation polynomial for equidistant base points in terms of the matrices $A_{s,\va}$:
\begin{theorem}\label{thm:1}
The interpolation problem $q_\va(x_i)=a_i$ associated to equidistant base points $x_i=\xi+ih, i=0,1,\ldots,\ell$ and the value vector $\va=(a_0,a_1,\ldots,a_\ell)$ yields a polynomial of degree  $\ell-m$ if and only if $\det A_{s,\va}=0$, for $s=0,\ldots,m-1$, and $\det A_{m,\va}\neq 0$.
\end{theorem}

When we were preparing this text we learned much about the relation between combinatorics and binomial identities and the evaluation of special determinants. We cordially refer the reader to the outstanding collections \cite{Kratten,Kratten2} of determinants evaluation. We also refer to the book \cite{Riordan} which presents special combinatorical identities.

\section{The proof of the main Theorem}

Theorem \ref{thm:1} relies on some properties of the matrix $A_{s,\va}$ and the polynomial $q_\va$ that can be obtained independently from each other. 
In this section we state these results and show how they prove Theorem \ref{thm:1}.

The first result is on the determinant of the matrix \eqref{matrixA}: 

\begin{proposition}\label{satz:2}
There is a constant $\sigma_\ell$ only depending on the size of the matrix $A_{s,\va}$ defined by \eqref{matrixA} such that its determinant is given by
\begin{equation}\label{eq:9}
\det(A_{s,\va}) = \sigma_\ell\, \sum_{j=0}^{\ell} (-1)^{j}   \binom{\ell}{j}j^s a_{j}\,.
\end{equation}
\end{proposition}

The second result deals with the derivatives of the interpolation polynomial \eqref{2}:

\begin{proposition}\label{satz:3}
For integers $s\geq 0$ and $0\leq k\leq s$ there are constants $\sigma_{\ell,s,k}$ such that  ($\ell-s$)-th derivative of the interpolation polynomial \eqref{2} for equidistant base points is given by 
\begin{equation}\label{eq:10} 
q_\va^{(\ell-s)}(\xi)= h^{s-\ell} 
\sum_{k=0}^s\sigma_{\ell,s,k} \bigg(\sum_{j=0}^\ell(-1)^{j}\binom{\ell}{j} j^k a_j\bigg)\,.
\end{equation}
\end{proposition}

Theorem \ref{thm:1} is a corollary of Propositions \ref{satz:2} and \ref{satz:3}:

\begin{proof}[Theorem \ref{thm:1}]
The interpolation polynomial  $q_\va(x)=\sum_{k=0}^\ell b_k(x-\xi)^k$ has degree $\ell-m$ if and only if $b_\ell=\cdots= b_{\ell-m+1}=0$ and $b_{\ell-m}\neq 0$. 
This is equivalent to have vanishing derivatives $q_\va^{(\ell-s)}(\xi)$ at the point $x=\xi$ for exactly $s=0,\ldots,m-1$. 
By \eqref{eq:10} with \eqref{eq:9} this is equivalent to $\det (A_{s,\va})=0$ for  $s=0,\ldots, m-1$ and $\det (A_{m,\va})\neq 0$.
\end{proof}

In the next sections we prove Proposition \ref{satz:2} and Proposition \ref{satz:3}.

\section{On the determinant of $A_{s,\va}$}

\subsection{Preliminaries}

Some submatrices of $A_{s,\va}$ \eqref{matrixA} are of special type and were studied in our previous work \cite{1}. They will be important in the reasoning of Proposition \ref{satz:2}. 

In \cite{1} we considered a positive integer $\ell$ and three sequences of complex numbers
$\bfx=(\alpha_1,\alpha_2,\ldots), \bfy=(\beta_1,\beta_2,\ldots)$, and	$\bfr=(r_1,r_2,\ldots)$. We assume $\bfr$ to be injective, i.e.\ $r_i\neq r_j$ for all $i,j\in\NN$. 
For each $k\in\NN$ these data define a complex $k\times k$-matrix $B:=B(k;\bfx,\bfy ,\bfr,\ell)\in M_k\CC$:
\begin{equation}\label{eq:main2}
B_{ij}=(\alpha_i+r_j\beta_i)^{\ell-1}\,.
\end{equation}
The determinant of this matrix is $\det(B)=0$  for $k>\ell$. For non vanishing  $\alpha_i,\beta_i$ and  
$\rho_i:=\frac{\beta_i}{\alpha_i}$ we obtain
\begin{equation}
\det(B)  = \prod_{i=1}^k \alpha_i^{\ell-1}\!\!\!\!  \sum_{0\leq \mu_1<\ldots<\mu_{k}\leq\ell-1}\, \prod_{j=1}^k{\binom{\ell-1}{\mu_j}} 
	 V_{k,\mu}(r_1,\ldots,r_k)\, V_{k,\mu}\big(\rho_1,\ldots,\rho_k\big)\label{detAneu}
\end{equation}
or
\begin{multline}\label{eq:1}
\hspace*{-0.8cm}\det(B) =\\ \hspace*{-0.8cm}(-1)^{\frac{k(k-1)}{2}}\prod_{i=1}^k \beta_i^{\ell-1} \sum_{0\leq \mu_1<\ldots<\mu_{k}\leq\ell-1}\, \prod_{j=1}^k{\binom{\ell-1}{\mu_j}} 	V_{k,\mu}(r_1,\ldots,r_k)\, V_{k,\mu^\complement}\big(\tfrac{1}{\rho_1},\ldots,\tfrac{1}{\rho_k}\big)
\tag{\ref{detAneu}'}
\end{multline}
for $k\leq \ell$, where 
$\mu$ is a sequence $\mu=(\mu_1,\ldots,\mu_k)$ of strictly increasing natural numbers $0\leq \mu_1< \ldots <\mu_k\leq \ell-1$, 
$\mu^\complement$ is the associated strictly increasing sequence $\mu^\complement=(\ell-1-\mu_k,\ldots,\ell-1-\mu_1)$, and   
\begin{equation}
V_{k,\mu}(\nu_1,\ldots,\nu_k)=\det\big(\nu_i^{\mu_j}\big)
\end{equation}
denotes the generalized Vandermonde determinant, see \cite{Heinemann}. 
In particular, $V_{k,(0,1,\ldots,k-1)}=V_k$ where $V_k$ is the usual Vandermonde determinant with $V_k(\nu_1,\ldots,\nu_k)=\prod\limits_{1\leq i<j\leq k}(\nu_j-\nu_i)$. 
Since $V_{k,\mu}(\nu_1,\ldots\nu_k)$ vanishes whenever $\nu_i=\nu_j$ the quotient ${V_{k,\mu}}/{V_k}$ is a polynomial, too. These symmetric polynomials are the Schur polynomials, see e.g.~\cite{DeMarchi,King,Littlewood}. 

A consequence of this discussion is
\begin{theorem}[{\cite[Theorem 3]{1}}]
Let $\bfx,\bfy$ be sequences of complex numbers with\footnote{In \cite{1} we show that the regularity result remains valid when one of the $\alpha_i$ or $\beta_i$ vanishes.} $\frac{\alpha_i}{\beta_i}\in\RR^+$ and $\alpha_i\beta_j-\beta_i\alpha_j\neq 0$ for all $1\leq i,j\leq k$, and $\bfr$ be an positive injective sequence.
Then the matrix \eqref{eq:main2} is regular if and only if $k\leq \ell$.
\end{theorem}

\subsection{The matrix $A_{s,\va}$ and its determinant}

Let us consider the matrix $A_{s,\va}\in M_{\ell+1}\RR$ defined by \eqref{matrixA}. Its components are given by\footnote{We use the convention $0^0=1$ to cover all cases by this formula.} 
\begin{equation*}\label{matrixA'}\begin{aligned}
&(A_{s,\va})_{ij} =\begin{cases}
\big(\alpha_i+r_j\beta_i \big)^{\ell-1} &\text{ for }\ 1\leq i\leq \ell, 1\leq j\leq \ell+1\\
(j-1)^sa_{j-1} & \text{ for }\ i=\ell+1, 1\leq j\leq \ell+1
\end{cases}\\[2ex]
& \text{with }\  \alpha_i:=(i-1)(\ell+1),\quad \beta_i=1,\quad  r_j=j \,.
\end{aligned}
\end{equation*}
Furthermore, denote by $A^{[\kappa]}\in M_\ell\RR$ the matrix obtained from $A_{s,\va}$ by removing the last row and the $\kappa$-th column. Of course, it does not depend on $s$ and $\va$ and is given by 
\begin{equation}
\begin{aligned}
&(A^{[\kappa]})_{ij} =\big(\alpha_i+r^{[\kappa]}_j\beta_i\big)^{\ell-1}\\[2ex]
& \text{ with }\   
r^{[\kappa]}_j=
\begin{cases}
j &\text{ for }\ 1\leq j\leq \kappa-1\,,\\ 
j+1&\text{ for }\  \kappa\leq j\leq \ell\,.
\end{cases}
\end{aligned}
\end{equation}
The matrices $A^{[\kappa]}$ are of the form \eqref{eq:main2} and have the following property:
\begin{proposition}\label{satz:7}
There exists a number $\sigma_\ell$ depending only on $\ell$ such that 
\begin{equation}
\det(A^{[\kappa]})=(-1)^\ell\sigma_\ell \,\binom{\ell}{\kappa-1} \,.
\end{equation}
\end{proposition}

An immediate corollary is Proposition \ref{satz:2}:

\begin{proof}[Proposition \ref{satz:2}]
Expanding the determinant of $A_{s,\va}$ with respect to the last row yields
\begin{align*}
\det (A_{s,\va}) &= \sum_{j=1}^{\ell+1}(-1)^{\ell+1+j} A_{\ell+1,j}\det A^{[j]}\\
%&= (-1)^{\ell}\sum_{j=1}^{\ell+1}(-1)^{j-1} (j-1)^sa_{j-1} \det A^{[j]}\\
&= (-1)^{\ell}\sum_{j=0}^{\ell}(-1)^{j} j^sa_{j} \det A^{[j+1]}\\
&=  \sigma_\ell \sum_{j=0}^{\ell} (-1)^{j}  \binom{\ell}{j} j^s a_{j} \,.
\end{align*}
\end{proof}

In the sequel we prove Proposition \ref{satz:7}:

\begin{proof}[Proposition \ref{satz:7}] 
We use \eqref{eq:1} for $k=\ell$ and note that only one summand is left, namely the one with $\mu=\mu^\complement=(0,1,\ldots\ell-1)$. In this case we have  $V_{\ell,\mu}=V_{\ell,\mu^\complement}=V_\ell$ and, therefore, 
\begin{align*}
\hspace*{-95cm}\det(A^{[\kappa]}) 
=\ & (-1)^{\frac{\ell(\ell-1)}{2}} 
	\prod_{i=1}^\ell \beta_i^{\ell-1} 
	\prod_{j=1}^\ell \binom{\ell-1}{j-1}  
	V_\ell\big(\tfrac{\alpha_1}{\beta_1},\ldots,\tfrac{\alpha_\ell}{\beta_\ell}\big)
	V_\ell(r_1^{[\kappa]},\ldots,r^{[\kappa]}_\ell)\\
=\ &(-1)^{\frac{\ell(\ell-1)}{2}}
	\prod_{j=0}^{\ell-1} \binom{\ell-1}{j} 
	\prod_{1\leq j_1<j_2\leq \ell}\!\!(\alpha_{j_2}-\alpha_{j_1})
	\prod_{1\leq i_1<i_2\leq \ell} \!\!(r_{i_2}^{[\kappa]}-r^{[\kappa]}_{i_1})\\
=\ &(-1)^{\frac{\ell(\ell-1)}{2}}(\ell+1)^{\frac{\ell(\ell-1)}{2}}
	\prod_{j=0}^{\ell-1} \binom{\ell-1}{j}\prod_{1\leq j_1<j_2\leq \ell}(j_2-j_1) %\,\cdot \\
	%&\ \ \cdot
	\prod_{1\leq i_1<i_2\leq \ell} (r_{i_2}^{[\kappa]}-r^{[\kappa]}_{i_1})  \\
=\ &(-1)^{\frac{\ell(\ell-1)}{2}} 
	(\ell+1)^{\frac{\ell(\ell-1)}{2}}
	\prod_{j=0}^{\ell-1} \binom{\ell-1}{j} \prod_{j=1}^{\ell-1}j!
	\prod_{1\leq i_1<i_2\leq \ell} (r_{i_2}^{[\kappa]}-r^{[\kappa]}_{i_1})\,.
\end{align*}
The proof is completed if we show that the quotient of 
$
\prod\limits_{1\leq i_1<i_2\leq \ell} (r_{i_2}^{[\kappa]}-r^{[\kappa]}_{i_1})$ and $ \binom{\ell}{\kappa-1}$ only depends  on $\ell$.
For $j>i$ we have 
\[
r_j^{[{\kappa}]}-r_i^{[\kappa]} 
=\begin{cases}
j-i &\text{ if } 1\leq i<j\leq\kappa-1\text{ or } \kappa\leq i<j\leq\ell\\
j-i+1 &\text{ if } 1\leq i\leq \kappa-1, \kappa\leq j\leq\ell\,,
\end{cases}
\]
and, therefore, 
\begin{align*}
\prod_{1\leq i_1<i_2\leq \ell}\!\! (r_{i_2}^{[\kappa]}-r^{[\kappa]}_{i_1})
&=\Bigg(\prod_{j=1}^{\kappa-1}\prod_{i=1}^{j-1}(j-i)\Bigg)\Bigg(\prod_{j=\kappa}^{\ell}\prod_{i=\kappa}^{j-1}(j-i)\Bigg)\Bigg(\prod_{j=\kappa}^{\ell}\prod_{i=1}^{\kappa-1}(j-i+1)\Bigg)\\
&= \prod_{i_1=1}^{\kappa-2}i_1! \, \prod_{i_2=1}^{\ell-\kappa}i_2! \, \prod_{j=\kappa}^{\ell}\prod_{i=1}^{\kappa-1}(j-i+1)\\
&= \prod_{i_1=1}^{\kappa-2}i_1! \, \prod_{i_2=1}^{\ell-\kappa}i_2! \, \prod_{m=1}^{\ell-\kappa+1}\prod_{n=1}^{\kappa-1}(m+n)\\
&= \prod_{i_1=1}^{\kappa-2}i_1! \, \prod_{i_2=1}^{\ell-\kappa}i_2! \, \prod_{p=2}^{\ell}p^{\gamma_{p}}\,.
\end{align*}
Here $\gamma_p$ denotes the number of factors with value $p$ in the third product, i.e.\
\[
\gamma_p=
\begin{cases}
p-1 & 2\leq p\leq \kappa\,,\\
\kappa-1 & \kappa+1\leq p\leq \ell-\kappa+1\,,\\
\ell-p+1 &\ell-\kappa+2\leq p\leq \ell\,.
\end{cases}
\]
Rearranging the factors, we obtain
\begin{align*}
\prod_{1\leq i_1<i_2\leq \ell}\!\! (r_{i_2}^{[\kappa]}-r^{[\kappa]}_{i_1})
&= \prod_{i_1=1}^{\kappa-2}i_1! \, \prod_{i_2=1}^{\ell-\kappa}i_2! \, \prod_{q=1}^{\kappa-1} \frac{(\ell-q+1)!}{q!}\\
&= \prod_{i_1=1}^{\kappa-2}i_1! \,\prod_{i_3=1}^{\kappa-1}\frac{1}{i_3!} \, \prod_{i_2=1}^{\ell-\kappa}i_2! \cdot \prod_{i_4=\ell-\kappa+2}^\ell i_4! \\
&= \frac{1}{(\kappa-1)!} \, \prod_{i_2=1}^{\ell-\kappa}i_2!\,\prod_{i_4=\ell-\kappa+2}^\ell i_4! \\
&= \frac{1}{(\kappa-1)!} \, \frac{\ell!}{(\ell-\kappa+1)!}\,\prod_{j=1}^{\ell-1} j! \\
&= \binom{\ell}{\kappa-1}\prod_{j=1}^{\ell-1} j!\,.
\end{align*}
We write 
\[
\sigma_\ell:=(-1)^{\frac{\ell(\ell+1)}{2}} 
	(\ell+1)^{\frac{\ell(\ell-1)}{2}}
	\prod_{j=0}^{\ell-1} \binom{\ell-1}{j} \prod_{j=1}^{\ell-1}(j!)^2
\]
and finally achieve 
\[
\det(A^{[\kappa]}) = (-1)^\ell \sigma_\ell\binom{\ell}{\kappa-1}\,.
\]
\end{proof}

\begin{remark}
The proof of Proposition 3 can also be performed by reducing the determinant at hand to a matrix with polynomial entries 
\[
p^{[\kappa]}_j(i)  = a^{[\kappa]}_j i^{j-1} +\text{l.o.t.} 
\qquad\text{ or }\qquad
p_i(r_j^{[\kappa]})  = a_i (r_j^{[\kappa]})^{i-1}+\text{l.o.t.}\,.
\]
Then applying the determinant evaluation as known from the Vandermonde matrix one can show that the determinant is given by 
\[
\Big(\prod_{j=1}^{\ell-1} a_j^{[\kappa]}\Big)\, V_\ell(1,\ldots,\ell)
\qquad\text{ or }\qquad
\Big(\prod_{i=1}^{\ell-1} a_i \Big)\,V_\ell(r_1^{[\kappa]},\ldots,r_\ell^{[\kappa]})\,,
\] 
see \cite[Propostion 1]{Kratten}.
However, in both cases the explicit calculations above remain to be done.
\end{remark}

\section{On the interpolant $q_\va$}

In this section we study the interpolation polynomial $q_\va$ \eqref{2} associated to equidistant base points and a value vector $\va$, see e.g.\ \cite{numeric,numeric2}. 

Explicitly, we have 
\begin{equation*}
q_\va(x)=\sum_{j=0}^{\ell} a_j L_j(x)
\end{equation*}
where $L_j$ denotes the  $j$-th Lagrange polynomial associated to the base points $x_0=\xi,x_1=\xi+ h,\ldots, x_\ell=\xi+\ell h$. For 
$j=0,\ldots ,\ell$ the latter are given by 
\begin{equation*}\begin{aligned}
L_j(x)
&=\frac{\prod_{i=0,i\neq j}^\ell (x-x_i)}{\prod_{i=0,i\neq j}^\ell (x_j-x_i)}\,.
\end{aligned}\end{equation*}

The proof of Proposition \ref{satz:3} will be presented at the end of this section and it makes use of a combinatorical fact on the polynomial
\begin{equation}\label{eq:K}
K(t)=\prod_{i=0}^\ell (t-i)\,.
\end{equation}
To formulate this fact, we introduce the symbols
\begin{equation*}
\tau_{\ell,m,j}=
\begin{cases}
\displaystyle 
\qquad 1 & \text{ for }m= 0\,,\\[1ex]
\displaystyle 
\sum_{1\leq i_1<\ldots<i_m\leq \ell}  i_1\cdots i_m &\text{ for }0<m\leq \ell,j= 0\,,\\[3ex]
\displaystyle 
\sum_{\overset{\scriptstyle 1\leq i_1<\ldots<i_m\leq \ell}{i_1,\ldots,i_m\neq j}}  i_1\cdots i_m &\text{ for }0<m\leq \ell-1,j>0\,,\\[3ex]
\qquad 0 & \text{ for }m= \ell,j>0\,.
\end{cases}
\end{equation*}
For instance, $\tau_{\ell,1,j}=\frac{\ell(\ell+1)}{2}-j$ for $0\leq j\leq \ell$ and $\tau_{\ell,\ell,0}=\ell!$, $\tau_{\ell,\ell-1,j} =\frac{\ell!}{j}$ for $1\leq j\leq \ell$. 
\begin{proposition}\label{prop:6}
For $0\leq j\leq \ell$ the polynomial $K(t)$ of degree $\ell+1$ obeys
\begin{equation}\label{eq:cam}
K(t) = ({t-j})\sum_{m=0}^\ell (-1)^m \tau_{\ell,m,j} t^{\ell-m}\,.
\end{equation}
Furthermore, for $j>0$ and $m\neq \ell$ the symbols $\tau_{\ell,m,j}$ obey
\begin{equation}\label{tau}
\tau_{\ell,m,j} = \sum_{k=0}^m (-1)^k\tau_{\ell,m-k,0} j^k\,.
\end{equation}
\end{proposition}

\begin{proof}
The first part is obtained by expanding of $\frac{K(t)}{t-j}=\prod\limits_{\overset{i=0}{i\neq j}}^\ell(t-i)$.
The second part is obviously true for $j=0$. For $j> 0$ we note
\begin{align*}
\tau_{\ell,m,j}=\ & 
\sum_{\overset{\scriptstyle 1\leq i_1<\ldots<i_m\leq \ell}{i_1,\ldots,i_m\neq j}}  i_1\cdots i_m\\
=\ &\sum_{\overset{\scriptstyle 1\leq i_1<\ldots<i_m\leq \ell}{}}  i_1\cdots i_m
- j\sum_{k=1}^m
\sum_{\overset{\scriptstyle 1\leq i_1<\ldots<i_{m}\leq\ell}{i_k=j}}  i_1\cdots i_{k-1}i_{k+1} \cdots i_m\\
=\ &\tau_{\ell,m,0}
- j\sum_{k=1}^{m}
\sum_{\overset{\scriptstyle 1\leq i_1<\ldots<i_{k-1}<j}{j< i_{k+1}<\ldots<i_{m}\leq\ell}}  i_1\cdots i_{k-1}i_{k+1} \cdots i_m\\
=\ &\tau_{\ell,m,0}
- j\sum_{k=1}^{m}
\sum_{\overset{\scriptstyle 1\leq \bar i_1<\ldots<\bar i_{k-1}<j}{j< \bar i_{k}<\ldots<\bar i_{m-1}\leq\ell}}  \bar i_1 \cdots \bar i_{m-1}\\
=\ &\tau_{\ell,m,0}
-j\sum_{\overset{\scriptstyle 1\leq \bar i_1<\ldots<\bar i_{m-1}\leq \ell}{\bar i_1,\ldots,\bar i_{m-1}\neq j}}  \bar i_1 \cdots \bar i_{m-1}\\
 =\ &\tau_{\ell,m,0} - j\tau_{\ell,m-1,j}\,.
\end{align*}
Using this recurrence, \eqref{tau} holds by induction.
\end{proof}

\begin{proof}[Proposition \ref{satz:3}]

We define $x=\xi+th$, $\hat L_j(t):= L_j(\xi+th)$, and $\hat q_\va(t):=q_\va(\xi+th)$, that is 
\begin{align*}
\hat L_j(t)
&=\frac{\prod_{i=0,i\neq j}^\ell h(t-i)}{\prod_{i=0,i\neq j}^\ell h(j-i)}
=\frac{\prod_{i=0,i\neq j}^\ell (t-i)}{\prod_{i=0,i\neq j}^\ell (j-i)} =\frac{(-1)^{\ell-j}}{\ell!}\binom{\ell}{j}\, \frac{K(t)}{t-j}\,
\end{align*}
and
\begin{equation*}
\hat q_\va(t)=\sum_{j=0}^\ell  a_j \hat L_j(t)=\frac{(-1)^{\ell}}{\ell!} \sum_{j=0}^\ell (-1)^j\binom{\ell}{j}a_j  \frac{K(t)}{t-j}
\end{equation*}
with
$K(t)$ defined in \eqref{eq:K}.  
By Proposition \ref{prop:6}  we have
\begin{align}\label{eq:14}
\hat q_\va(t)=\ &
\frac{(-1)^{\ell}}{\ell!} 
\sum_{m=0}^\ell\sum_{k=0}^m (-1)^{k+m}\tau_{\ell,m-k,0} \bigg(\sum_{j=0}^\ell(-1)^{j}\binom{\ell}{j} j^k a_j\bigg) 
 t^{\ell-m}
\end{align}
and
\begin{multline*}
\hspace*{-0.8cm}\hat q^{(\ell-s)}_\va(t) \\
\hspace*{-0.8cm} =  \frac{(-1)^{\ell}(\ell-s)!}{\ell!} 
\sum_{m=0}^s \binom{\ell-m}{s-m} \sum_{k=0}^m (-1)^{k+m} \tau_{\ell,m-k,0} \bigg(\sum_{j=0}^\ell(-1)^{j}\binom{\ell}{j} j^k a_j\bigg)
t^{s-m}\,.
\end{multline*}
In particular, we obtain  at $t=0$ 
\begin{align*}
\frac{(-1)^{\ell-s}\ell!}{(\ell-s)!} \hat q^{(\ell-s)}_\va(0)=\ & \sum_{k=0}^s (-1)^{k}\tau_{\ell,s-k,0} \bigg(\sum_{j=0}^\ell(-1)^{j}\binom{\ell}{j} j^k a_j\bigg)\,.
\end{align*}
For instance, for the first values of $s$ the right hand side of this expression is:
\begin{align*}
\hspace*{-0.3cm} s=0: &\ 
\tau_{\ell,0,0}\sum_{j=0}^\ell(-1)^{j}\binom{\ell}{j} a_j\,,
\\
\hspace*{-0.3cm} s=1: &\ 
\tau_{\ell,1,0} \sum_{j=0}^\ell(-1)^{j}\binom{\ell}{j}  a_j
-\tau_{\ell,0,0} \sum_{j=0}^\ell(-1)^{j}\binom{\ell}{j} j a_j\,,
\\
\hspace*{-0.3cm} s=2: &\  
\tau_{\ell,2,0} \sum_{j=0}^\ell(-1)^{j}\binom{\ell}{j} a_j
-\tau_{\ell,1,0} \sum_{j=0}^\ell(-1)^{j}\binom{\ell}{j} j a_j
+\tau_{\ell,0,0} \sum_{j=0}^\ell(-1)^{j}\binom{\ell}{j} j^2 a_j\,.
\end{align*}
Writing
\[
\sigma_{\ell,s,k}:= \frac{(-1)^{\ell-s+k}(\ell-s)!}{\ell!}\tau_{\ell,s-k,0}
\]
and using $\hat q_\va^{(\ell-s)} (\xi+th) = h^{\ell-s}q_\va^{(\ell-s)}(t)$ we obtain \eqref{eq:10}.
\end{proof}

\section{A final remark on the general interpolation problem}

\begin{remark}
In the  case of arbitrary base points $x_0,\ldots,x_\ell$ one can introduce 
modified symbols $\tau_{\ell,m,j}=\sum\limits_{\overset{\scriptstyle 1\leq i_1<\ldots<i_m\leq \ell}{i_1,\ldots,i_m\neq j}}  x_{i_1}\cdots x_{i_m}$
to get an analogous of \eqref{eq:cam}.

In \cite{Camargo} these symbols are denoted by $\widehat{e}_{\ell-m,j}$ and are used to present representations of the Schur functions.
In our setting, they can be used to get an expansion like \eqref{tau}, too, and a formula of the form  \eqref{eq:14} can also be obtained:
\[
 q_\va(x)=\ 
\sum_{m=0}^\ell\sum_{k=0}^m (-1)^{k+m}\tau_{\ell,m-k,0} \bigg(\sum_{j=0}^\ell \lambda_j x_j^k a_j\bigg) 
 x^{\ell-m}
\]
with $\lambda_j^{-1}=\prod\limits_{\overset{\scriptstyle i=0}{i\neq j}}^\ell(x_i-x_j)$. However, a result as in Theorem \ref{thm:1} is not obtained in such a nice form.
\end{remark}

%\begin{acknowledgements}
%We thank the anonymous referee for his comments that helped to improve the readability of our %text.
%\end{acknowledgements}


\begin{thebibliography}{99}

\bibitem{DeMarchi}
De Marchi, S.:
\newblock{Polynomials arising in factoring generalized Vandermonde determinants: an algorithm for computing their coefficients.}
\newblock {\em Math.\ Comput.\ Modelling} {\bf 34} no.~3/4, 271-281 (2001)

\bibitem{Camargo}
de Camargo A.~P.:
\newblock{Schur functions through Lagrange polynomials.}
\newblock{J.\ Pure Appl.\ Algebra} {\bf 220} no.~8, 2948-2954 (2016) 

\bibitem{numeric2}
Davis, P.~J.:
\newblock{\em Interpolation and Approximation.}
\newblock Dover Publication, Inc, New York, 1975 

\bibitem{numeric}
Deuflhard, P.\ and Hohmann, A.:
\newblock{\em Numerical Analysis in Modern Scientific Computing.}
\newblock Springer-Verlag New York, Inc., 2003

\bibitem{Heinemann}
Heineman, E.~R.: 
\newblock{Generalized Vandermonde determinants.}
\newblock{\em Trans.\ Amer.\ Math.\ Soc.} {\bf 31} no.~3, 464-476 (1929) 

\bibitem{King}
King, R.~C.:
\newblock{Generalised Vandermonde determinants and Schur functions.}
\newblock{\em Proc.\ Amer.\ Math.\ Soc.} {\bf 48} no.~1,  53-56 (1975)

\bibitem{1}
Klinker, F.\ and Reineke, C.:
\newblock{On the regularity of matrices with uniform polynomial entries.}
\newblock{\em S\~ao Paulo J.~Math.~Sci.} (2018), online 2017 

\bibitem{Kratten}
Krattenthaler, C.:
\newblock{Advanced determinant calculus.}
\newblock{\em S\'em.\ Lothar.\ Combin.} {\bf 42} Art.~B42q, 67pp (1999)

\bibitem{Kratten2}
Krattenthaler, C.: 
\newblock{Advanced determinant calculus: a complement.} 
\newblock{\em Linear Algebra Appl.} {\bf 411}, 68-166 (2005)

\bibitem{Littlewood}
Littlewood, D.~E.: 
\newblock{\em The Theory of Group Characters and Matrix Representations of Groups.}
\newblock Oxford University Press, New York, 1940

\bibitem{Riordan}
Riordan, J.:
\newblock{\em Combinatorical Identities.}
\newblock Reprint of the 1968 original. Robert E. Krieger Publishing Co., Huntington, N.Y., 1979


 

\end{thebibliography}
\end{document}